\documentclass[10pt]{amsart}
\usepackage{amssymb, amsmath}
\usepackage{graphics}
\usepackage{latexsym}
\usepackage{amsmath}
\usepackage{amssymb,amsthm,amsfonts}
\usepackage{amscd}
\usepackage[arrow, matrix, curve]{xy}
\usepackage{syntonly}
\usepackage{yfonts}
\usepackage{tikz-cd}
\usepackage{filecontents}
\usepackage{mathtools}
\usepackage{stmaryrd}
\usepackage{MnSymbol}
\DeclareMathOperator{\spec}{Spec}
\DeclareMathOperator{\gal}{Gal}

\DeclareMathOperator{\br}{Br}
\DeclareMathOperator{\ram}{Ram}

\theoremstyle{plain}
\newtheorem{thm}{Theorem}[section]
\newtheorem{theorem}[thm]{Theorem}
\newtheorem{lemma}[thm]{Lemma}
\newtheorem{remcor}[thm]{Remark/Corollary}

\theoremstyle{definition}

\newtheorem{remark}[thm]{Remark}

\numberwithin{equation}{thm}
\newcommand{\sL}{{\mathcal L}}

\newcommand{\sO}{{\mathcal O}}

\newcommand{\sT}{{\mathcal T}}

\newcommand{\C}{{\mathbb C}}

\renewcommand{\H}{{\mathbb H}}

\renewcommand{\P}{{\mathbb P}}
\newcommand{\Q}{{\mathbb Q}}
\newcommand{\R}{{\mathbb R}}

\newcommand{\V}{{\mathbb V}}

\newcommand{\Z}{{\mathbb Z}}

\newcommand{\rk}{{\rm rank}}

\begin{document}
\title{On Shimura curves generated by families of Galois $G$-covers of curves}
\author{Abolfazl Mohajer}
\address{School of Mathematical and Statistical Sciences, University of Galway, Galway, Ireland.}
\email{abmohajer83@gmail.com}
\maketitle
\begin{abstract}
In this paper we prove there are no families of cyclic $\Z_n$-covers of elliptic curves which generate non-compact Shimura (special) curves that lie generically in the Torelli locus  $T_g$ of abelian varieties with $g\geq 8$ when $n$ has a proper prime factor $p\geq 7$. This non-existence is also shown for families of $\Z_n$-covers of curves of any genus $s$ and when $n$ has a large enough prime number $p$ (depending on $s$). We achieve these results by applying the theory of Higgs bundles and the Viehweg-Zuo characterization of Shimura curves in the moduli space of principally polarized abelian varieties.
\end{abstract}
\section{Introduction}
Let $M_g$ be the moduli space of complex non-singular algebraic curves. The \emph{Torelli} or \emph{period} morphism $j: M_g\to A_g$ associates to the isomorphism class of a curve $[C]$ the isomorphism class of its Jacobian $[J(C)]$. This morphism is injective by a celebrated result of Torelli. We call the image  
$T^{\circ}_g$ the \emph{open Torelli locus}. The Zariski closure $T_g$ of $T^{\circ}_g$ inside $A_g$ is called the \emph{Torelli locus}. A conjecture by Coleman, based on the analogy with the Manin-Mumford conjecture, says that for $g\geq 4$ there are only finitely many complex curves of genus g such that $J(C)$ has complex multiplication. Combining with the Andr\'e-Oort Theorem, this implies that for $g$ large enough, there are no special subvarieties $Z\subseteq A_g$ of positive dimension that are contained inside $T_g$ and intersect $T^{\circ}_g$ non-trivially. Coleman's conjecture is known to be false for $\{4,5,6,7\}$ by the work of many authors, see the works of de Jong-Noot \cite{DN} and Viehweg-Zuo \cite{VZ05}. Some examples of genera $g=5,7$ were found by Rohde in \cite{R}. All these examples were constructed as families of cyclic coverings of $\P^1$. Some other examples have been constructed in \cite{FGP} and \cite{FPP}. In these two papers, P. Frediani and M. Penegini et al. found examples of non-abelian covers of $\P^1$ and also of coverings of elliptic curves which give rise to Shimura varieties in $T_g$. There has also been a lot of effort in the opposite direction: in proving that some loci do not contain any Shimura subvariety. The first results here are due to Moonen who proved in \cite{M1} that there are precisely 20 families giving rise to Shimura subvarieties in the Torelli locus. We extended this result to all abelian covers of $\P^1$  in \cite{MZ}. A significant progress has been made through results of Lu and Zuo in \cite{LZ} and Chen, Lu and Zuo in \cite{CLZ}. In these papers the authors use techniques from theory of Higgs bundles developed for instance in \cite{VZ04} to exclude large classes of special subvarieties from the Torelli locus and in particular from the Torelli locus of cyclic $G$-covers of $\P^1$.\par In the present paper we follow the approach of these papers and use Higgs bundles in order to exclude non-compact Shimura curves from the Torelli locus of $\Z_p$-covers of elliptic curves when $p\geq 7$. If the genus of the base curve is $s>1$, our methods can be used to prove that there does not exist families of genus $g$ curves of $\Z_p$-covers of curves of genus $s$ giving rise to non-compact Shimura curves when $p$ is a large enough prime number depending on $s$, see Theorem \ref{non-comp p gen}. This holds also for families of $\Z_n$-covers, when $n$ is a number having large enough prime factor $p$ (again depending on $s$), see Theorem \ref{non-comp n gen}. In section 2 we review some facts about Galois covers of curves and their families which will be used to analyze such families later on in the paper. In section 3 we introduce Higgs bundles and their connection to Shimura curves in the moduli space of abelian varieties and prove our main results.

\section{Galois coverings of curves}
Fix a smooth curve $C^{\prime}$ of genus $g$, suppose that we have a positive integer $d$ and a divisor $D=\displaystyle\sum_{k=1}^ru_kP_k$ on $C^{\prime}$ such that $1\leq u_k<n$ and that $n\divides\displaystyle\sum_{k=1}^ru_k$. Furthermore suppose that $\sL$ is a line bundle on $C^{\prime}$ such that $D$ is the divisor of zeroes of a section of $\sL^n$. To each such datum $(\sL, n, D)$ we can associate a cyclic covering $C\to C^{\prime}$ of degree $n$ such that there are $\gcd(n,u_k)$ points above the point $P_k$ on $C$. Let us define the line bundles $\sL^{(i)}=\sL^i(-\sum\lfloor\frac{iu_k}{n}\rfloor P_k)$ for any $0\leq i<n$ on the curve $C$. Then $C=\spec_{\sO_{C^{\prime}}}(\oplus {\sL^{(i)}}^{-1})$.  In particular, if in the branch divisor $D, u_k=1$, then the cover $C\to C^{\prime}$ is totally ramified. There is only one point over $P_k$ on $C$ necessarily of ramification index $d$. Note that in this case we also have $\sL^{(i)}=\sL^i$. Conversely, to a $n$-sheeted cyclic cover $C\to C^{\prime}$, we can associate a line bundle $\sL$ such that $\sL^n=\sO_{C^{\prime}}(D)$. \par Note that if $C\to C^{\prime}$ is a cyclic cover with covering group $G\cong \Z/n\Z$, then $G$ acts on $C$ and hence also on the space of global holomorphic 1-forms $H^0(C,K_C)$. Let $H^0(C,K_C)=\oplus H^0(C,K_C)_i$ be the decomposition of $H^0(C,K_C)$ into $G$-eigenspaces as a $G$-representation. The dimension of the subvector space $H^0(C,K_C)_i$ can be computed by the well-known Chevalley-Weil formula:

\begin{theorem}{(Chevalley-Weil)}\label{CW}
 $h_i=\dim_{\C}H^0(C,K_C)_i=g-1+\displaystyle\sum_{k=1}^r\langle\frac{-iu_k}{n}\rangle$ for $1\leq i\leq n-1$ and $h_0=g$.\par
 In the special case where the cover is totally ramified, i.e., $u_k=1$ for every $k$, then $h_i=g-1+r(1-\frac{i}{n})$. This happens for example if $n=p$ is a prime number. 
\end{theorem}
In the  above theorem, $\langle\cdotp\rangle$ denotes the fractional part of a real number.

\section{Shimura curves and Higgs bundles}
\subsection{Shimura varieties}
In this subsection, we will breifly introduce and review some facts about Shimura varieties and special subvarieties. In this note we only consider connected Shimura varieties which are complex algebraic varieties of the form $\Gamma\backslash X^+$, where $X^+$ is a Hermitian symmetric domain and $\Gamma$ is a congruence subgroup of a semisimple algebraic group $G^{der}(\Q)$ acting on $X^+$. Inside $\Gamma\backslash X^+$ there are Shimura or special subvarieties associated with Shimura subdata, see \cite{Mil}, \S 5. In simple terms, they are subvarieties of the form $\Gamma\backslash {X^{\prime}}^+\subseteq \Gamma\backslash X^+$, where ${X^{\prime}}^+\subseteq X^+$ are equivariant embeddings of Hermitian symmetric subdomains of $X^+$. Equivariantly embedded here means that it is defined by some semi-simple Lie subgroup $G^{\prime}$ of $G^{der}(\R)$ and the inclusion ${X^{\prime}}^+\hookrightarrow X^+$ is equivariant with respect to $G^{\prime}\hookrightarrow G^{der}(\R)$. The zero-dimensional special subvarieties are CM points. In this paper, we are interested about higher dimensional special subvarieties. The most important Shimura variety for us is the moduli space of complex $g$-dimensional principally polarized abelian varieties $A_g$. Here $A_g=A_{g,l}=\Gamma_g(l)\backslash \H_g$, where $\H_g$ is the Siegel uper half-space of genus $g$ and $\Gamma_g(l)$ is the principal congruence subgroup of level-$l$ in $Sp_{2g}(\Z)$ which is the kernel of the natural map $Sp_{2g}(\Z)\to Sp_{2g}(\Z/l)$, where $l\geq 3$ is an odd integer. In this case $\Gamma_g(l)$ is torsion-free and the quotient $\Gamma_g(l)\backslash \H_g$ is a smooth complex submanifold. Special subvarities of dimension one are called \emph{Shimura curves}.\\

In this paper we investigate Shimura curves in $A_g$ inside the Torelli locus and cuting the locus $T^{\circ}_g$ non-trivially. As indicated in the introduction, the Coleman-Oort conjecture asserts that for $g$ large enough, there are no special subvarieties  inside the Torelli locus and cutiing the locus $T^{\circ}_g$ non-trivially. In this note we investigate this conjecture for subvarieties coming from families of Galois $G$-covers of curves. Indeed we consider families whose fibers are $C_t\to C_t/G$ which give rise to subvarieties in $M_g$ and, by the Torelli map, in $A_g$. Any such variety naturally lies in the Torelli locus and intersects $T^{\circ}_g$. 

\subsection{Higgs bundles}
Let $\overline{f}:\overline{S}\to\overline{B}$ be a family of semi-stable curves representing a Shimura curve $C$ in the Torelli locus $\sT_g$. We further suppose that the smooth fibers of the family are $G$-covers $D\to D/G=D^{\prime}$ and there are no hyperelliptic fibers. Note that, as it is mentioned in \cite{CLZ}, Remark 3.(iii) after a finite base change if necessary, the $G$-action on the fibers induces a $G$-action on the surface $\overline{S}$ which restricts to the $G$-action on the fibers. It is this action on $\overline{S}$ that we use in this paper to appply the theory of cyclic covers. Let us observe some properties of this family: To this family is associated a Higgs bundle with local system $\V_B:=R^1f_*\Q_{\overline{S}\setminus\Delta}$. This Higgs bundle is $(E^{1,0}_{\overline{B}}\oplus E^{0,1}_{\overline{B}}, \theta_{\overline{B}})$ in which
\[E^{1,0}_{\overline{B}}=\overline{f}_*\Omega_{\overline{S}/\overline{B}}, E^{0,1}_{\overline{B}}=R^1\overline{f}_*\sO_{\overline{S}}\]
and the Higgs field is
\[\theta_{\overline{B}}:E^{1,0}_{\overline{B}}\to E^{0,1}_{\overline{B}}\otimes\Omega_{\overline{B}}(\log\Delta_{nc}) \]
There is a decomposition 
\[(E^{1,0}_{\overline{B}}\oplus E^{0,1}_{\overline{B}}, \theta_{\overline{B}})=(A^{1,0}_{\overline{B}}\oplus A^{0,1}_{\overline{B}}, \theta_{\overline{B}}|_{A^{1,0}_{\overline{B}}})\oplus (F^{1,0}_{\overline{B}}\oplus F^{0,1}_{\overline{B}}, 0)\]
where $A^{1,0}_{\overline{B}}$ is ample and $F^{1,0}_{\overline{B}}\oplus F^{0,1}_{\overline{B}}$ is flat and corresponds to a unitary local subsystem $\V^u_B\subset \V_B\otimes\C$. Since we assumed that there are no hyperelliptic fibers in our family, we have the following characterization, see \cite{VZ04}
\begin{equation}
C \text{ is a Shimura curve }\Leftrightarrow \deg E^{1,0}_{\overline{B}}=\frac{\deg A^{1,0}_{\overline{B}}}{2}\deg\Omega^1_{\overline{B}}(\log\Delta_{nc})
\end{equation}
If C is a non-compact Shimura curve, then by \cite{LZ}, \S 3, 
\begin{equation}
g(\overline{F})=\rk F^{1,0}_{\overline{B}} \hspace{1 cm} \text{ for any fiber } \overline{F} \text{ over } \Delta_{nc},
\end{equation}
where $g(\overline{F})$ denotes the geometric genus of $\overline{F}$.\\

Since there is a $G$-action on the surface $\overline{S}$, we have an eigenspace decomposition 
\[\V_B\otimes\C=\bigoplus_{i=0}^{n-1}\V_{B,i};  (E^{1,0}_{\overline{B}}\oplus E^{0,1}_{\overline{B}}, \theta_{\overline{B}})=\bigoplus_{i=0}^{n-1}(E^{1,0}_{\overline{B}}\oplus E^{0,1}_{\overline{B}}, \theta_{\overline{B}})_i\]
and furthermore
\[g=\sum_{i=0}^{n-1}\rk E^{1,0}_{\overline{B},i}\]
and also 
\[\rk E^{0,1}_{\overline{B},i}=\rk E^{0,1}_{\overline{B},n-i}\]

The eigenspace decomposition is compatible with this decomption and we have
\begin{align}
&(A^{1,0}_{\overline{B}}\oplus A^{0,1}_{\overline{B}}, \theta_{\overline{B}})=\bigoplus_{i=0}^{n-1}(A^{1,0}_{\overline{B}}\oplus A^{0,1}_{\overline{B}}, \theta_{\overline{B}}|_{A^{1,0}_{\overline{B}}})_i,\\
&(F^{1,0}_{\overline{B}}\oplus F^{0,1}_{\overline{B}}, 0)=\bigoplus_{i=0}^{n-1}(F^{1,0}_{\overline{B}}\oplus F^{0,1}_{\overline{B}}, 0)_i
\end{align}

Each local subsystem $\V_{\overline{B},i}$ is defined over the n-th cyclotomic field $\mathbb{Q}(\xi_n)$ and so the arithmetic Galois group $\gal(\mathbb{Q}(\xi_n)/\mathbb{Q})$ has a natural action on he above decompositions
\begin{lemma}
Let $\overline{f}:\overline{S}\to\overline{B}$ be a family of semi-stable $G$-covers representing a Shimura curve $C$ in the Torelli locus $\sT_g$. Let $\V^{tr}_{\overline{B},i}\subset\V_{\overline{B},i}$ be the trivial local subsystem, and $((F^{1,0}_{\overline{B},i})^{tr}\oplus (F^{0,1}_{\overline{B},i})^{tr}, 0)$ be the associated trivial flat subbundle. If $\V_{\overline{B},i}$ and $\V_{\overline{B},j}$ are in one  $\gal(\mathbb{Q}(\xi_n)/\mathbb{Q})$-orbit, then
\begin{align}
&\rk \V^{tr}_{\overline{B},i}=\rk \V^{tr}_{\overline{B},j}\\
&\rk (F^{1,0}_{\overline{B},i})^{tr}+\rk (F^{1,0}_{\overline{B},n-i})^{tr}=\rk (F^{1,0}_{\overline{B},j})^{tr}+\rk (F^{1,0}_{\overline{B},n-j})^{tr}
\end{align}
In particular, if $n=p$ is a prime number, for any $1\leq i<j\leq p-1$, 
\begin{align}
&\rk \V^{tr}_{\overline{B},i}=\rk \V^{tr}_{\overline{B},j}\\
&\rk (F^{1,0}_{\overline{B},i})^{tr}+\rk (F^{1,0}_{\overline{B},n-i})^{tr}=\rk (F^{1,0}_{\overline{B},j})^{tr}+\rk (F^{1,0}_{\overline{B},n-j})^{tr}
\end{align}
\end{lemma}
\begin{proof}
This is exactly the same as \cite{CLZ}, Lemma 3.2. Indeed the proof of the aforementioned result is independent of the fact that the fibers are superelliptic and goes through for any family of $n$-cyclic covers of curves. 
\end{proof}

\begin{lemma}\label{higgsrek}
Let $\overline{f}:\overline{S}\to\overline{B}$ be a family of semi-stable $G$-covers representing a Shimura curve $C$ in the Torelli locus $\sT_g$. Then
\begin{equation}\label{rkeqs}
\rk A^{1,0}_{\overline{B},i}=\rk A^{0,1}_{\overline{B},i}=\rk A^{1,0}_{\overline{B},n-i} \forall 1\leq i\leq n-1
\end{equation}
For $i=0$, we have
\begin{equation}
\rk A^{1,0}_{\overline{B},0}=\rk A^{0,1}_{\overline{B},0}
\end{equation}
In particular, 
\begin{gather}
\rk F^{1,0}_{\overline{B},i}\neq 0 \hspace{2cm}\text{ if        }\rk E^{1,0}_{\overline{B},i}>\rk E^{1,0}_{\overline{B},n-i};\label{rknzero}\\
\rk F^{1,0}_{\overline{B},n-i}\geq \rk F^{1,0}_{\overline{B},i} \hspace{2cm}\text{      for        } i\geq n/2 \label{rkgreater}
\end{gather}
\end{lemma}
\begin{proof}
Since $C$ is a Shimura curve, there is a decomposition 
\[(E^{1,0}_{\overline{C}}\oplus E^{0,1}_{\overline{C}}, \theta_{\overline{C}})=(A^{1,0}_{\overline{C}}\oplus A^{0,1}_{\overline{C}}, \theta_{\overline{C}}|_{A^{1,0}_{\overline{C}}})\oplus (F^{1,0}_{\overline{C}}\oplus F^{0,1}_{\overline{C}}, 0)\]
in which the Higgs field $\theta_{\overline{C}}|_{A^{1,0}_{\overline{C}}}$ is an isomorphism. Note that $(A^{1,0}_{\overline{B}}\oplus A^{0,1}_{\overline{B}}, \theta_{\overline{B}}|_{A^{1,0}_{\overline{B}}})$ is the pull-back of $(A^{1,0}_{\overline{C}}\oplus A^{0,1}_{\overline{C}}, \theta_{\overline{C}}|_{A^{1,0}_{\overline{C}}})$ under an isomorphism and hence $\theta_{\overline{B}}|_{A^{1,0}_{\overline{B}}}$ must also be an isomorphism. Since the Higgs field respects the eigenspace decomposition, by restricting to $(A^{1,0}_{\overline{B}}\oplus A^{0,1}_{\overline{B}}, \theta_{\overline{B}}|_{A^{1,0}_{\overline{B}}})_i$ we obtain that $\theta_{\overline{B}}|_{A^{1,0}_{\overline{B},i}}$ is an isomorphism. This proves that 
$\rk A^{1,0}_{\overline{B},i}=\rk A^{0,1}_{\overline{B},i}$ as claimed. The second equality follows by complex conjugating. Furthermore, \ref{rkeqs} implies that 
\[\rk E^{1,0}_{\overline{B},i}-\rk E^{1,0}_{\overline{B},n-i}=\rk F^{1,0}_{\overline{B},i}-\rk F^{1,0}_{\overline{B},n-i}, 1\leq i\leq n-1\]
or equivalently,
\[\rk F^{1,0}_{\overline{B},i}=\rk F^{1,0}_{\overline{B},n-i}+(\rk E^{1,0}_{\overline{B},i}-\rk E^{1,0}_{\overline{B},n-i}),  1\leq i\leq n-1\]
from which \ref{rknzero} follows. Finally, for $i\geq n/2$, again using the above equality together with Theorem \ref{CW} yields
\[ \rk F^{1,0}_{\overline{B},n-i}-\rk F^{1,0}_{\overline{B},i}=\rk E^{1,0}_{\overline{B},n-i}-\rk E^{1,0}_{\overline{B},i}=\frac{r(2i-n)}{n}\geq 0\]
which implies \ref{rkgreater}. 
\end{proof}

\begin{lemma}\label{second fib}
Let $\overline{f}:\overline{S}\to\overline{B}$ be a family of semi-stable $G$-covers representing a non-compact Shimura curve $C$ in the Torelli locus $\sT_g$ and either
\begin{enumerate} 
\item $|G|=n\geq 4$ is an even number, or
\item $|G|=n\geq 3$ is an odd number
\end{enumerate}
Further, assume that after a suitable base change of $\overline{B}$, there exists a fibration $\overline{f}^{\prime}:\overline{S}\to\overline{B}^{\prime}$ such that $g(\overline{B}^{\prime})\geq  \rk F^{1,0}_{\overline{B}} $. Then $g<8$. 
\end{lemma}
\begin{proof}
Assume that such a fibration $\overline{f}^{\prime}$ with the above properties exists but $g\geq 8$. Then restricting $\overline{f}^{\prime}$ to an arbitrary fiber $\overline{F}$ of $\overline{f}$ we get a morphism $\overline{f}^{\prime}|_{\overline{F}}:\overline{F}\to \overline{B}^{\prime}$. Note that $\deg(\overline{f}^{\prime}|_{\overline{F}})$ does not depend on choice of the fiber $\overline{F}$. Since the fibration $\overline{f}$ is non-isotrivial, we have that $\deg(\overline{f}^{\prime}|_{\overline{F}})\geq 2$. Consequently, by the Riemann-Hurwitz formula we have
\begin{equation}\label{RH}
2g(\overline{F})-2\geq 2(2g(\overline{B}^{\prime})-2)
\end{equation}
Furthermore, by Lemma ~\ref{higgsrek} and Theorem ~\ref{CW}, we have
\begin{equation}\label{ineqrk}
\rk F^{1,0}_{\overline{B}}\geq\displaystyle\sum_{i=1}^{\lfloor n/2\rfloor}\rk F^{1,0}_{\overline{B},i}\geq\sum_{i=1}^{\lfloor n/2\rfloor}(\rk E^{1,0}_{\overline{B},i}-\rk E^{0,1}_{\overline{B},i})=\begin{cases}
\frac{r(n-2)}{4} & \text{ if $n$ is even }\\
\frac{r(n-1)^2}{4n} & \text{ if $n$ is odd }
\end{cases}
\end{equation}
Now the above assumptions (1) and (2) on $n$ and the fact that $g\geq 8$ imply that both of the numbers in the last equality in \ref{ineqrk} are greater than $2$. Since $C$ is non-compact, if in ~\ref{RH} we choose a fiber $\overline{F}$ over $\Delta_{nc}$ we get a contradiction to [] because $g(\overline{B}^{\prime})\geq  \rk F^{1,0}_{\overline{B}}\geq 2$. 
\end{proof}

\begin{remcor}\label{second fib big n}
\begin{enumerate} 
\item From the proof of Lemma \ref{second fib}, one sees that the result readily follows if $\rk F^{1,0}_{\overline{B}}\geq 2$. In fact this is a more powerful statement generalizing many other results in this direction. More precisely, the author has proved in \cite{M} that if a family of abelian covers of $\P^1$ has a eigenspace $\sL_{\chi}$ such that $d_{\chi}\geq 2$  and $d_{\chi^{-1}}\geq 2$, then the subvariety arising from this family is not Shimura. More results in this direction can also be found in the paper \cite{F} by P. Frediani. Also Moonen proves similar results about cyclic covers of $\P^1$, see \cite{M1}, (7.3).   
\item  If $n\geq 10$, then by the calculations in ~\ref{ineqrk}, it holds that $\rk F^{1,0}_{\overline{B}}\geq 2$. Consequently, Lemma \ref{second fib} holds if $n\geq 10$. 
\end{enumerate}
\end{remcor}

Let $\overline{f}:\overline{S}\to\overline{B}$ be a family of semi-stable $\Z_n$-covers of curves. We will assume that the $\Z_n$-action extends to $\overline{S}$ and remark that this can always be achieved using suitable finite covers. The Galois cover $\overline{\Pi}:\overline{S}\to\overline{Y}$ induces a Galois cover $\overline{\Pi}^{\prime}:S^{\prime}\to\widetilde{Y}$ whose branch locus is a normal crossing divisor and with $\gal(\overline{\Pi}^{\prime})\cong G$, where $\widetilde{Y}\to\overline{Y}$ is the minimal resolution of singularities. Let $\widetilde{S}\to S^{\prime}$ be the minimal resolution of
singularities. Then there is an induced birational contraction $\widetilde{S}\to \overline{S}$. Since $\Pi^{\prime}$ is a $\Z_n$-cover of surfaces, it is defined by a relation $\sL^n\equiv\sO_Y(R)$. \\

The following Lemma has been proven in \cite{CLZ}, Lemma 4.9 for general families of covers

\begin{lemma}\label{trivial base change}
Let $\widetilde{R}\subseteq\widetilde{Y}$ be the reduced branch divisor of $\Pi^{\prime}$ as above, and $\widetilde{\Gamma}$ be a general fiber of $\widetilde{\varphi}$. If $\widetilde{R}$ contains at least one section of $\widetilde{\varphi}$, and there exist $1\leq i_1\leq i_2\leq p-1$ such that $\rk F^{1,0}_{\overline{B}, i_1}\neq 0, \rk F^{1,0}_{\overline{B}, i_2}\neq 0$ and
\begin{equation}\label{inter}
\widetilde{\Gamma}\cdotp(\omega_{\widetilde{Y}}(\widetilde{R})\otimes ({\mathcal{L}^{(i_1)}}^{-1}\otimes {\mathcal{L}^{(i_2)}}^{-1}))<0
\end{equation}
Then both $F^{1,0}_{\overline{B}, i_1}, F^{1,0}_{\overline{B}, i_2}$ become trivial after a suitable finite \'etale base change.
\end{lemma}
In the following result and thereafter we assume that we have a family of $\Z_p$-covers of elliptic curves. Later on we will generalize this result partially for any families of $\Z_p$-covers of curves of any genus $s>1$, see Theorems \ref{non-comp p gen} and \ref{non-comp n gen}. 

\begin{lemma}\label{inter neg}
Suppose that our family is a family of $\Z_p$-covers of elliptic curves. Assume also that $\widetilde{R}$ contains at least one section of $\widetilde{\varphi}$ and that $\rk F^{1,0}_{\overline{B},i_0}\neq 0$ for some $i_0\geq p+1/2$. Then after a suitable unramified base change, $F^{1,0}_{\overline{B},i}$ becomes trivial for any $p-i_0+1\leq i< i_0$. 
\begin{proof}
First of all note that $\widetilde{\Gamma}\cdotp\omega_{\widetilde{Y}}(\widetilde{R})= \deg \widetilde{R}|_{\widetilde{\Gamma}}$. Therefoere by definition of $\sL$, it follows that 
\[\widetilde{\Gamma}\cdotp(\omega_{\widetilde{Y}}(\widetilde{R})\otimes ({\mathcal{L}^{(i)}}^{-1}\otimes {\mathcal{L}^{(i_0)}}^{-1}))<0, \forall p-i_0+1\leq i< i_0.\]
The claim then follows from Lemma ~\ref{trivial base change}. 
\end{proof}
\end{lemma}
The following lemma, follows from the general proof of Lemma 4.12 of \cite{CLZ}.
\begin{lemma}\label{inter neg}
Assume that there exist $1\leq i_1\leq i_2\leq p-1$, such that
\begin{equation}\label{inter neg0}
\widetilde{\Gamma}\cdotp(\omega_{\widetilde{Y}}(\widetilde{R})\otimes ({\mathcal{L}^{(i_1)}}^{-1}\otimes {\mathcal{L}^{(i_2)}}^{-1}))<0
\end{equation}
and such that $H^0(\overline{S}, \Omega^1_{\overline{S}})_{i_1}\neq 0$ and $H^0(\overline{S}, \Omega^1_{\overline{S}})_{i_2}\neq 0$. Then there exists a unique fibration $\overline{f}^{\prime}:\overline{S}\to \overline{B}^{\prime}$ such that 
\begin{equation}\label{pull-back gen}
H^0(\overline{S}, \Omega^1_{\overline{S}})_{i}\subseteq (\overline{f}^{\prime})^* H^0(\overline{B}^{\prime}, \Omega^1_{\overline{B}^{\prime}}) \text{ for } i=i_1, i_2
\end{equation}
\end{lemma}

Before we state one of our main theorems, we prove the following lemma which is fundamental for our further results.

\begin{lemma}\label{rank neq0}
Assume that $p\geq 5$ then 
\begin{equation}\label{non-triviality}
\rk F^{1,0}_{\overline{B}, i}>0
\end{equation}
for any $\frac{p+1}{2}\leq i<p-1$
\end{lemma}
\begin{proof}
Since this claim is independent of the base change, we may assume by \cite{VZ04}, Corollary 4.4 that the unitary local subsystem $\V_B^u\subseteq \V_B\otimes\C$ is trivial after a suitable finite base change, i.e., $\V_B^u=\V_B^{tr}$. By the above, we have
\[\rk (F^{1,0}_{\overline{B},i})+\rk (F^{0,1}_{\overline{B},i})=\rk (F^{1,0}_{\overline{B},j})+\rk (F^{0,1}_{\overline{B},j}), \forall 1\leq i\leq j\leq p-1\]
One can also check that 
\[\rk (E^{1,0}_{\overline{B},i})+\rk (E^{0,1}_{\overline{B},i})=\rk (E^{1,0}_{\overline{B},j})+\rk (E^{0,1}_{\overline{B},j}), \forall 1\leq i\leq j\leq p-1\]
Now this together with \ref{rkeqs} yields
\begin{equation}
\rk A^{1,0}_{\overline{B},i}=\rk A^{0,1}_{\overline{B},j}, \forall 1\leq i\leq j\leq p-1
\end{equation}
Hence 
\begin{align}
& \rk F^{1,0}_{\overline{B},i}=\rk E^{1,0}_{\overline{B},i}-\rk A^{1,0}_{\overline{B},i}\\
&= \rk E^{1,0}_{\overline{B},i}-\rk A^{1,0}_{\overline{B},p-1}= \rk E^{1,0}_{\overline{B},i}-(\rk E^{1,0}_{\overline{B},p-1}-\rk F^{1,0}_{\overline{B},p-1})\\ 
&\geq \rk E^{1,0}_{\overline{B},i}-\rk E^{1,0}_{\overline{B},p-1}=\frac{r(p-(i+1))}{p}>0
\end{align}
The last inequality holds for the chosen $i$ and for $p\geq 5$.
\end{proof}

\begin{lemma}\label{second fib p}
%Choose some $i_0>p+1/2$ so that $\rk F^{1,0}_{\overline{B},i_0}\neq 0$ by Lemma \ref{rank neq0}. 
Assume $p\geq 7$, then after suitable base change, there exists a fibration $\overline{f}^{\prime}:\overline{S}\to \overline{B}^{\prime}$ such that 
\[\displaystyle\oplus_{i=0}^{p-1} H^0(\overline{S}, \Omega^1_{\overline{S}})_i\subseteq  (\overline{f}^{\prime})^* H^0(\overline{B}^{\prime}, \Omega^1_{\overline{B}^{\prime}})\]
and hence $g(\overline{B}^{\prime})\geq \rk F^{1,0}_{\overline{B}}$. 
\end{lemma}
\begin{proof}
Choose some $\frac{p+1}{2}< i_0<p-1$, so that $\rk F^{1,0}_{\overline{B}, i_0}>0$ by Lemma \ref{rank neq0}. We choose also an $i$ such that $p-i_0+1\leq\frac{p+1}{2}\leq i< i_0$. Note that $\rk F^{1,0}_{\overline{B}, i}>0$ again by Lemma \ref{rank neq0}. Then the eigenspaces $\mathcal{L}^{(i)}, \mathcal{L}^{(i_0)}$ satisfy
\begin{equation}\label{inter neg1}
\widetilde{\Gamma}\cdotp(\omega_{\widetilde{Y}}(\widetilde{R})\otimes ({\mathcal{L}^{(i)}}^{-1}\otimes {\mathcal{L}^{(i_0)}}^{-1}))<0
\end{equation}
Furthermore we have $\rk F^{1,0}_{\overline{B}, p-i}>0$ by virtue of \ref{rkgreater} and also
\begin{equation}\label{inter neg2}
\widetilde{\Gamma}\cdotp(\omega_{\widetilde{Y}}(\widetilde{R})\otimes ({\mathcal{L}^{(p-i)}}^{-1}\otimes {\mathcal{L}^{(i_0)}}^{-1}))<0
\end{equation}
Now  Lemma ~\ref{inter neg} implies that there exist unique fibrations $\overline{f}_j^{\prime}:\overline{S}\to \overline{B}_j^{\prime}$ such that
\[H^0(\overline{S}, \Omega^1_{\overline{S}})_{i_0}\oplus H^0(\overline{S}, \Omega^1_{\overline{S}})_{j}\subseteq (\overline{f}_j^{\prime})^*H^0(\overline{B}_j^{\prime}, \Omega^1_{\overline{B}_j^{\prime}})\]
for $j=i, p-i$. Since $H^0(\overline{S}, \Omega^1_{\overline{S}})_{i_0}\neq 0$, the uniqueness of $\overline{f}_j^{\prime}$ implies that these morphisms are the same which we denote by $\overline{f}^{\prime}:\overline{S}\to\overline{B}^{\prime}$. This shows that $H^0(\overline{S}, \Omega^1_{\overline{S}})_i\subseteq (\overline{f}^{\prime})^*H^0(\overline{B}^{\prime}, \Omega^1_{\overline{B}^{\prime}})$ and  $H^0(\overline{S}, \Omega^1_{\overline{S}})_{p-i}\subseteq (\overline{f}^{\prime})^*H^0(\overline{B}^{\prime}, \Omega^1_{\overline{B}^{\prime}})$. Therefore
\[H^0(\overline{S}, \Omega^1_{\overline{S}})_i\oplus H^0(\overline{S}, \Omega^1_{\overline{S}})_{p-i}\subseteq (\overline{f}^{\prime})^*H^0(\overline{B}^{\prime}, \Omega^1_{\overline{B}^{\prime}})\]
Now the pull-back map $(\overline{f}^{\prime})^*:H^1(\overline{B}^{\prime}, \Q)\to H^1(\overline{S}, \Q)$ is defined over $\Q$ and we have 
\[\overline{H^0(\overline{S}, \Omega^1_{\overline{S}})_{p-i}}=H^1(\overline{S}, \sO_{\overline{S}})_{i}\]
Hence
\[H^0(\overline{S}, \Omega^1_{\overline{S}})_i\oplus H^1(\overline{S}, \sO_{\overline{S}})_{i}=(H^1(\overline{S}, \Q)\otimes\C)_i\subseteq  (\overline{f}^{\prime})^*(H^1(\overline{B}^{\prime}, \Q)\otimes\C)\]
Note that the arithmetic Galois group $\gal(\Q(\xi_p)/\Q)$ acts on the eigenspace decomposition 
\[H^1(\overline{S}, \Q\otimes\C)=H^1(\overline{S}, \Q)\otimes\C)_0\displaystyle\bigoplus_{i=1}^{p-1} (H^1(\overline{S}, \Q)\otimes\C)_i\]
The eigenspaces $(H^1(\overline{S}, \Q)\otimes\C)_i$ for $1\leq i\leq p-1$ are permuted by this action and the eigenspace $H^1(\overline{S}, \Q)\otimes\C)_0$ is fixed by this action. Since this action is transitive, it holds that \par $\displaystyle\bigoplus_{i=1}^{p-1} (H^1(\overline{S}, \Q)\otimes\C)_i\subseteq  (\overline{f}^{\prime})^*(H^1(\overline{B}^{\prime}, \Q)\otimes\C)$. Note that in any case,\par $H^0(\overline{S}, \Omega^1_{\overline{S}})_0\subseteq  (\overline{f}^{\prime})^*H^0(\overline{B}^{\prime}, \Omega^1_{\overline{B}^{\prime}})$. Therefore we have that
\[H^1(\overline{S}, \Q)\otimes\C\subseteq  (\overline{f}^{\prime})^*(H^1(\overline{B}^{\prime}, \Q)\otimes\C)\]
and hence $g(\overline{B}^{\prime})\geq \rk F^{1,0}_{\overline{B}}$. 
\end{proof}

\begin{theorem}\label{non-comp}
Let $p\geq 7$ be a prime number. Then there does not exist any non-compact Shimura curve contained generically in the Torelli locus of $p$-cyclic covers of elliptic curves of genus $g\geq 8$.
\end{theorem}
\begin{proof}
Assume on the contrary that there exists a Shimura curve $C$ in the above locus. Let $\overline{f}:\overline{S}\to\overline{B}$ be a family of semi-stable $p$-cyclic covers of elliptic curves which represents $C$. Possibly after a base change we may assume that the group $G\cong \Z/p\Z$ on $\overline{S}$ and therefore also on the Higgs bundle $(E^{1,0}_{\overline{B}}\oplus E^{0,1}_{\overline{B}}, \theta_{\overline{B}})$ and its subbundles with eigenspace decompositions. We first of all treat the case where $C$ is non-compact. In this case by \cite{VZ04}, Corollary 4.4 the flat subbundle $F^{1,0}_{\overline{B}}$ becomes trivial after a suitable base change, i.e., $F^{1,0}_{\overline{B}}\cong\sO^{\oplus s}_{\overline{B}}$, where $s=\rk F^{1,0}_{\overline{B}}$. Furthermore, by Lemma~\ref{second fib p} that there exists another fibration $\overline{f}^{\prime}:\overline{S}\to\overline{B}^{\prime}$ different from $\overline{f}$ such that $g(\overline{B}^{\prime})\geq \rk F^{1,0}_{\overline{B}}$. But this contradicts Lemma \ref{second fib} which proves our claim that $g< 8$. 
\end{proof}

\begin{remark}
The above results can be generalized to families of $\Z_p$-covers of curves of any genus. The generalization of all of the results is straightforward (and indeed already independent of the genus of the base curve) except Lemma \ref{inter neg}. Note that in the general case, where we consider the covers of curves of genus $s$, we have
\begin{equation}\label{inter neg general}
\widetilde{\Gamma}\cdotp\omega_{\widetilde{Y}}(\widetilde{R})=2s-2+\deg \widetilde{R}|_{\widetilde{\Gamma}}
\end{equation}
In this case the prime number $p$ must be large enough so that we can take $p-i_0+a(s)+1\leq i<i_0$ as in Lemma \ref{second fib p} for some number $a(s)$ depending on $s$. Therefore we have the following result
\end{remark}
\begin{theorem}\label{non-comp p gen}
Let $s>0$ be a positive integer. Then there exists and integer $b(s)$ such that for all  prime numbers $p\geq b(s)$ there does not exist any non-compact Shimura curve contained generically in the Torelli locus $\sT_g$ of $p$-cyclic covers of curves of genus $s$ of genus $g\geq 8$.
\end{theorem}

If $n$ is not a prime we can prove also statements like Theorem \ref{non-comp}. Before doing this, we need some auxiliary results
\begin{lemma} \label{ramification of quotient}
Let $f:C\to Y$ be a Galois $G$-cover of degree $|G|=n$ and $N\lhd G$ a normal subgroup. Consider the following diagram in which $\overline{f}:\overline{C}\to Y$ is the quotient cover with the notations as above. 
\[\begin{tikzcd}
C \arrow{rr}{f} \arrow[swap]{dr}{\varphi} & & Y \\
& \overline{C} \arrow[swap]{ur}{\overline{f}}
\end{tikzcd}\]
Then
\begin{enumerate}
\item  $\br(\overline{f})\subseteq \br(f)$, i.e., the branch points of $\overline{f}$ are also branch points of $f$. \item $\ram(\varphi)\subseteq \ram(f)$, i.e., the ramification points of $\varphi:C\to\overline{C}=C/N$ are also ramification points of $f$. \item $\ram(\overline{f})\subseteq \varphi(\ram(f))$, i.e., every ramification point of $\overline{f}$ is the image of a ramification point of $f$ under $\varphi$ and the inertia group of a ramification point $\overline{x}=\varphi(x)$ of $\overline{f}$ is (isomorphic to) $G_x/N\cap G_x$, where $G_x$ is the inertia group of the point $x\in C$. 
\end{enumerate}
\end{lemma}
Indeed the above statements can be summarized into the formula
\begin{equation}
\br(f)=\overline{f}(\br(\varphi))\cup \br(\overline{f})
\end{equation}

\begin{theorem}
Suppose that $n=pn^{\prime}$, where $p\geq 7$ and $n^{\prime}\geq 1$. Then there does not exist any non-compact Shimura curve contained generically in the Torelli locus of $n$-cyclic covers of elliptic curves with $r\geq 3$ branch points of genus $g\geq 8$.
\end{theorem}
\begin{proof}
Assume on the contrary that such a family exists. We consider the subgroup $H$ of $\Z_n$ generated by the element $n^{\prime}$. Consider the family of quotient covers $\overline{f}_1:\overline{S}/H\to\overline{B}$. The Galois group of the quoteint is isomorphic to $\Z_p$. In other words, we have a family of $p$-cyclic covers of elliptic curve. Note that since the covers are all totally ramified, the branch points of the quotient covers in $\overline{f}_1$ have also $r\geq 3$ is equal to $r=\br(f)$.  The Riemann-Hurwitz gives that the genus $g^{\prime}$ of the quotient curve is
\begin{equation}
g^{\prime}=1+\frac{1}{2}r(p-1)
\end{equation}
which gives that $g^{\prime}\geq 8$. This is in contradiction with Theorem \ref{non-comp}.
\end{proof}
Similar to the above Theorem ~\ref{non-comp p gen} we can also generalize the above theorem for families of curves of genus $s$.

\begin{theorem}\label{non-comp n gen}
Let $s>0$ be a positive integer. Then there exists an integer $c(s)$ such that if $n$ has a  prime factor $p\geq c(s)$ there does not exist any non-compact Shimura curve contained generically in the Torelli locus $\sT_g$ of $n$-cyclic covers of curves of genus $s$ of genus $g\geq 8$.
\end{theorem}

\end{document}